\documentclass[11pt,a4paper]{article}

\usepackage[utf8]{inputenc}
\usepackage[T1]{fontenc}
\usepackage{geometry}
\usepackage{amsmath,amssymb,amsthm}
\usepackage{hyperref}
\usepackage{listings}
\usepackage{xcolor}
\usepackage{enumitem}
\usepackage{booktabs}
\usepackage{array}

\usepackage{tikz}
\usepackage{pgf} 
\usetikzlibrary{tikzmark}
\usetikzlibrary{positioning}
\usetikzlibrary{patterns,arrows,decorations.pathreplacing}

\usepackage{biblatex}
\newtheorem{theorem}{Theorem}[section]

\newtheorem{lemma}[theorem]{Lemma}

\theoremstyle{definition}
\newtheorem{definition}[theorem]{Definition}
\newtheorem{example}[theorem]{Example}
\newtheorem{algorithm_env}[theorem]{Algorithm}
\newtheorem{remark}[theorem]{Remark}
\newtheorem{prop}[theorem]{Proposition}

\definecolor{codegray}{RGB}{248,248,248}
\definecolor{codeblue}{RGB}{0,60,140}
\definecolor{codegreen}{RGB}{0,110,0}
\definecolor{codered}{RGB}{140,0,0}

\lstdefinestyle{pseudocode}{
  backgroundcolor=\color{codegray},
  basicstyle=\ttfamily\small,
  keywordstyle=\color{codeblue}\bfseries,
  commentstyle=\color{codegreen}\itshape,
  frame=single,
  rulecolor=\color{gray!40},
  breaklines=true,
  showstringspaces=false,
  numbers=left,
  numberstyle=\tiny\color{gray},
  numbersep=8pt,
  tabsize=4,
  morekeywords={for, if, return, while, do, function, end, Input, Output, else},
}

\title{%
  Effective Bialynicki-Birula-Brosnan motivic decompositions
}

\author{Charles De Clercq\footnote{Université Sorbonne Paris Nord, LAGA, declercq@math.univ-paris13.fr}}

\date{}

\begin{document}

\maketitle

\begin{abstract}
Let $G$ be an isotropic reductive group and $X$ be a projective $G$-homogeneous variety. Using results from Bialynicki-Birula, Hesselink and Iversen, Brosnan showed that if $G$ is of inner type, the motive of $X$ can be expressed as a direct sum of Tate twists of motives of projective homogeneous varieties for the anisotropic kernel of $G$. We provide a SageMath implementation of this decomposition, based on a depth-first search of the Cayley graph of the Weyl group of $G$, ensuring the complexity scales with the size of the motive rather than the full Weyl group. As applications, we provide new motivic decompositions for some exceptional groups and show how to extend Karpenko's decompositions for classical groups to characteristic $2$.
\end{abstract}

\tableofcontents

\section{Introduction}

Rost decomposition of the motive of a projective quadric states that given an isotropic smooth projective $n$-dimensional quadric $Q$ over a field, its Chow motive $M(Q)$ admits a direct sum decomposition involving a shift of the motive of the anisotropic quadric associated to $Q$, and some Tate motives. Moreover, the shifts appearing in this decomposition are completely determined by the Witt index of the underlying quadratic form. The question of extending this kind of motivic decompositions to an arbitrary projective homogeneous variety for isotropic semi-simple groups, involving motives of projective homogeneous varieties with respect to anisotropic kernels, has been intensively studied in the last decades.

The first systematic results were obtained by Karpenko \cite{kar}, who provided explicit motivic decompositions for projective homogeneous varieties for classical groups, involving the underlying algebraic structures. A fundamental step was then achieved by Chernousov, Gille and Merkurjev, who constructed motivic decompositions for projective homogeneous varieties with a rational point \cite{chergilmer}. Among many other applications, this key result implies notably that projective homogeneous varieties satisfy Rost nilpotence.

Finally, using results of Bialynicki-Birula, Hesselink and Iversen, Brosnan provides with \cite[Theorem 7.4]{bros} explicit motivic decompositions for arbitrary projective homogeneous varieties for isotropic semi-simple groups. Given such a group $G$ and a projective $G$-homogeneous variety $X$, it gives a decomposition $$M(X)\simeq\bigoplus_{w\in \mathcal{W}} M(X_w)\{\alpha_w\}$$ of the motive of $X$, where $\mathcal{W}$ is a subset of the Weyl group of $G$ and the $X_w$'s are projective homogeneous varieties for the anisotropic kernel of $G$. Furthermore, the types of these varieties and the shifts involved in this decomposition are completely determined by the \emph{type} of the $G$-variety $X$, the \emph{anisotropic kernel} of $G$ and its \emph{Weyl group}.

However explicit, these Bialynicki-Birula-Brosnan motivic decompositions can be intractable by hand for huge Weyl groups (notably for semi-simple groups of exceptional types). This note aims at providing an efficient algorithm in SageMath performing these decompositions. The main computational solution consists of substituting to the naive approach of performing a bi-minimal reduction on all the elements of the Weyl group $W$ of $G$ an efficient algorithm. The key algorithmic improvement consists of an enumeration of right-minimal representatives through depth-first search (DFS) on the Cayley graph of $W$, followed by filtering a set of representatives for the additional left-minimality condition. In the second part we introduce the needed material to state Brosnan Theorem: semi-simple algebraic groups over fields and the associated combinatorial data, Tits indices, projective homogeneous varieties and motives. We then describe in sections 3 and 4 the main algorithm, with a complexity analysis. The last two parts are dedicated to the framework required to use our algorithm, and to provide some explicit examples of motivic decompositions it provides.

\section{Background on semi-simple algebraic groups and motives}
\label{sec:background}

\subsection{Semi-simple algebraic groups and projective homogeneous varieties}

Let $F$ be a field and $G$ be a semi-simple algebraic group over $F$ (we refer to \cite{kmrt} and \cite{milne} for general information about semi-simple algebraic groups). Denote by $F_{sep}$ a separable closure of $F$. Given a maximal torus $T\subset G$, denote by $\Phi(G)$ the root system of $G_{F_{sep}}$ with respect to $T_{F_{sep}}$ and by $\Delta_G$ the associated Dynkin diagram (obtained fixing a Borel subgroup of $G_{F_{sep}}$ containing $T_{F_{sep}}$). In the sequel, we often use $\Delta_G$ as well to denote the set of vertices of the Dynkin diagram of $G$, which correspond to simple roots $\Sigma$.

In this work, we only consider semi-simple algebraic groups of inner type, that is, the absolute Galois group $\mathrm{Gal}(F_{sep}/F)$ of $F$ acts trivially on the Dynkin diagram of $G$ through the $\ast$-action \cite[§27.A]{kmrt}. Although Bialynicki-Birula-Brosnan motivic decompositions hold in outer type as well, this case involves corestrictions over various extensions of $F$; encoding such data in an algorithmic input or output seems impractical.

A projective $G$-homogeneous variety $X$ is a $G$-variety, which is isomorphic over a separable closure of $F$ to a quotient of $G$ by a parabolic subgroup. As explained in \cite{boreltits} or \cite{indexred}, there is a bijection between the subsets of the set of vertices of the Dynkin diagram of $G$, and the isomorphism classes of projective $G$-homogeneous varieties. We will say that a projective $G$-homogeneous variety associated to a subset $\Theta$ of $\Delta_G$ is of type $\Theta$, and often write it $X_{\Theta,G}$.

\begin{remark}\label{convent}
There are two opposite conventions for the type of the projective homogeneous variety (see \cite{motequiv} and \cite{tits}). For clarity, we stick for now and until the last section with Tits convention used by Brosnan, so that the subset associated with the Borel variety of $G$ is the empty set. As explained later, our algorithm allows inputs with both conventions.
\end{remark}

Since $G$ is of inner type, the Galois action on its Dynkin diagram is trivial. Consequently, $G$ is isomorphic to a direct product $\prod_{i=1}^n G_i$ of absolutely simple groups, whose Dynkin diagrams are the connected components of $\Delta_G$. Any projective homogeneous variety $X_{\Theta,G}$ is therefore isomorphic to a product $X_{\Theta,G}\simeq \prod_{i=1}^n X_{\Theta_i,G_i}$, where the $\Theta_i$'s are the respective subsets of the Dynkin diagrams $\Delta_{G_1}$,...,$\Delta_{G_n}$ which cover $\Theta$. As for splitting off the motive of a projective homogeneous variety, in our case, it is enough to work with absolutely simple algebraic groups.\\\vspace{0.3cm}

\begin{table}[bth]
{\centering\noindent\makebox[492pt]{
\scalebox{1}{
\begin{tabular}[c]{p{2.2in}|p{3in}}

$\small{(A_n)~~}$
\begin{picture}(7,2)(0,0)
\put(0,1){\circle*{3}}
\put(0,1){\line(1,0){20}}
\put(20,1){\circle*{3}}
\put(20,1){\line(1,0){20}}
\put(40,1){\circle*{3}}
\put(40,-1.6){\hspace{6pt}\mbox{$\cdots$}}
\put(62,1){\circle*{3}}
\put(62,1){\line(1,0){20}}
\put(82,1){\circle*{3}}
\put(82,1){\line(1,0){20}}
\put(102,1){\circle*{3}}

\put(-2,-7){\mbox{\tiny $1$}}
\put(18,-7){\mbox{\tiny $2$}}
\put(38,-7){\mbox{\tiny $3$}}
\put(54,-7){\mbox{\tiny $n$$-$$2$}}
\put(75,-7){\mbox{\tiny $n$$-$$1$}}
\put(100,-7){\mbox{\tiny $n$}}
\end{picture}
\vspace{1.5cm}

&

$\small{(E_6)~~}$
\begin{picture}(7,2)(0,0)
\put(0,-5){\circle*{3}}
\put(0,-5){\line(1,0){15}}
\put(15,-5){\circle*{3}}
\put(15,-5){\line(1,0){15}}
\put(30,-5){\circle*{3}}
\put(30,10){\circle*{3}}
\put(30,-5){\line(0,1){15}}
\put(30,-5){\line(1,0){15}}
\put(45,-5){\circle*{3}}
\put(45,-5){\line(1,0){15}}
\put(60,-5){\circle*{3}}

\put(-2,-13){\mbox{\tiny $1$}}
\put(13,-13){\mbox{\tiny $3$}}
\put(28,-13){\mbox{\tiny $4$}}
\put(43,-13){\mbox{\tiny $5$}}
\put(58.5,-13){\mbox{\tiny $6$}}
\put(33,9){\mbox{\tiny $2$}}
\end{picture}
\vspace{0.7cm}
\\
$\small{(B_n)~~}$
\begin{picture}(7,2)(0,0)
\put(0,1){\circle*{3}}
\put(0,1){\line(1,0){20}}
\put(20,1){\circle*{3}}
\put(20,1){\line(1,0){20}}
\put(40,1){\circle*{3}}
\put(40,-1.6){\hspace{6pt}\mbox{$\cdots$}}
\put(62,1){\circle*{3}}
\put(62,1){\line(1,0){20}}
\put(82,1){\circle*{3}}
\put(82,2){\line(1,0){20}}
\put(82,0){\line(1,0){20}}
\put(89,-1){{\small\mbox{$>$}}}
\put(102,1){\circle*{3}}

\put(-2,-7){\mbox{\tiny $1$}}
\put(18,-7){\mbox{\tiny $2$}}
\put(38,-7){\mbox{\tiny $3$}}
\put(54,-7){\mbox{\tiny $n$$-$$2$}}
\put(75,-7){\mbox{\tiny $n$$-$$1$}}
\put(100,-7){\mbox{\tiny $n$}}
\end{picture}

&

$\small{(E_7)~~}$
\begin{picture}(7,2)(0,0)
\put(0,-5){\circle*{3}}
\put(0,-5){\line(1,0){15}}
\put(15,-5){\circle*{3}}
\put(15,-5){\line(1,0){15}}
\put(30,-5){\circle*{3}}
\put(30,-5){\line(1,0){15}}
\put(45,-5){\circle*{3}}
\put(45,-5){\line(1,0){15}}
\put(45,10){\circle*{3}}
\put(45,-5){\line(0,1){15}}
\put(60,-5){\circle*{3}}
\put(60,-5){\line(1,0){15}}
\put(75,-5){\circle*{3}}

\put(-2,-13){\mbox{\tiny $7$}}
\put(13,-13){\mbox{\tiny $6$}}
\put(28,-13){\mbox{\tiny $5$}}
\put(43,-13){\mbox{\tiny $4$}}
\put(58.5,-13){\mbox{\tiny $3$}}
\put(73,-13){\mbox{\tiny $1$}}
\put(47.5,9){\mbox{\tiny $2$}}
\end{picture}
\vspace{0.7cm}
\\
$\small{(C_n)~~}$
\begin{picture}(7,2)(0,0)
\put(0,1){\circle*{3}}
\put(0,1){\line(1,0){20}}
\put(20,1){\circle*{3}}
\put(20,1){\line(1,0){20}}
\put(40,1){\circle*{3}}
\put(40,-1.6){\hspace{6pt}\mbox{$\cdots$}}
\put(62,1){\circle*{3}}
\put(62,1){\line(1,0){20}}
\put(82,1){\circle*{3}}
\put(82,2){\line(1,0){20}}
\put(82,0){\line(1,0){20}}
\put(89,-1){{\small\mbox{$<$}}}
\put(102,1){\circle*{3}}

\put(-2,-7){\mbox{\tiny $1$}}
\put(18,-7){\mbox{\tiny $2$}}
\put(38,-7){\mbox{\tiny $3$}}
\put(54,-7){\mbox{\tiny $n$$-$$2$}}
\put(75,-7){\mbox{\tiny $n$$-$$1$}}
\put(100,-7){\mbox{\tiny $n$}}
\end{picture}
\vspace{0.5cm}

&

$\small{(E_8)~~}$
\begin{picture}(7,2)(0,0)
\put(0,-5){\circle*{3}}
\put(0,-5){\line(1,0){15}}
\put(15,-5){\circle*{3}}
\put(15,-5){\line(1,0){15}}
\put(30,-5){\circle*{3}}
\put(30,-5){\line(1,0){15}}
\put(45,-5){\circle*{3}}
\put(60,-5){\line(0,1){15}}
\put(60,-5){\circle*{3}}
\put(60,10){\circle*{3}}
\put(75,-5){\circle*{3}}
\put(75,-5){\line(1,0){15}}
\put(90,-5){\circle*{3}}
\put(45,-5){\line(1,0){15}}
\put(60,-5){\line(1,0){15}}

\put(-2,-13){\mbox{\tiny $8$}}
\put(13,-13){\mbox{\tiny $7$}}
\put(28,-13){\mbox{\tiny $6$}}
\put(43,-13){\mbox{\tiny $5$}}
\put(58.5,-13){\mbox{\tiny $4$}}
\put(73,-13){\mbox{\tiny $3$}}
\put(88,-13){\mbox{\tiny $1$}}
\put(62.5,9){\mbox{\tiny $2$}}
\end{picture}
\vspace{0.7cm}
\\

$\small{(D_n)~~}$
\begin{picture}(7,2)(0,0)
\put(0,1){\circle*{3}}
\put(0,1){\line(1,0){20}}
\put(20,1){\circle*{3}}
\put(20,1){\line(1,0){20}}
\put(40,1){\circle*{3}}
\put(40,-1.6){\hspace{6pt}\mbox{$\cdots$}}
\put(62,1){\circle*{3}}
\put(62,1){\line(1,0){20}}
\put(82,1){\circle*{3}}
\put(82,2){\line(4,3){15}}
\put(82,0){\line(4,-3){15}}
\put(96.5,12.9){\circle*{3}}
\put(96.5,-10.9){\circle*{3}}

\put(-2,-7){\mbox{\tiny $1$}}
\put(18,-7){\mbox{\tiny $2$}}
\put(38,-7){\mbox{\tiny $3$}}
\put(54,-7){\mbox{\tiny $n$$-$$3$}}
\put(86,-0.5){\mbox{\tiny $n$$-$$2$}}
\put(100,-12){\mbox{\tiny $n$}}
\put(100,11.8){\mbox{\tiny $n$$-$$1$}}
\end{picture}

&

$\small{(F_4)~~}$
\begin{picture}(7,2)(0,0)
\put(0,1){\circle*{3}}
\put(0,1){\line(1,0){15}}
\put(15,1){\circle*{3}}
\put(15,0){\line(1,0){15}}
\put(15,2){\line(1,0){15}}
\put(19,-1){{\small\mbox{$>$}}}
\put(30,1){\circle*{3}}
\put(30,1){\line(1,0){15}}
\put(45,1){\circle*{3}}

\put(-2,-7){\mbox{\tiny $1$}}
\put(13,-7){\mbox{\tiny $2$}}
\put(28,-7){\mbox{\tiny $3$}}
\put(43,-7){\mbox{\tiny $4$}}

\put(65,1){$\small{(G_2)~~}$}
\put(92,1){\circle*{3}}
\put(92,0.1){\line(1,0){15}}
\put(92,1.1){\line(1,0){15}}
\put(92,2.1){\line(1,0){15}}
\put(96,-1){{\small\mbox{$<$}}}
\put(107,1){\circle*{3}}

\put(90,-7){\mbox{\tiny $1$}}
\put(105,-7){\mbox{\tiny $2$}}
\end{picture}
\vspace{0.5cm}
\end{tabular}
}}}
\caption{Connected Dynkin diagrams, with simple roots numbered} \label{dynk.table}
\end{table}

\subsection{Anisotropic kernels and Tits indexes}

Let $G$ be a semi-simple algebraic group over $F$. The \emph{rank} of $G$ is the dimension of a maximal split torus. A semi-simple algebraic group is isotropic if it contains a non-trivial split torus, and anisotropic otherwise.

\begin{definition}
    Let $G$ be a semi-simple group and $T_0\subset G$ be a maximal split torus. The derived group of the centralizer of $T_0$ in $G$ is the anisotropic kernel of $G$, denoted by $G_{an}$. 
\end{definition}

The anisotropic kernel of $G$ is an anisotropic semi-simple algebraic group, which does not depend on the choice of a maximal split torus \cite[Prop. 1.9]{kerreh}.

The Tits index is a fundamental invariant used in the classification of semi-simple groups to encode their isotropy. Given a semi-simple algebraic group $G$, it consists of the Dynkin diagram $\Delta_G$, endowed with a fixed subset of vertices $\Theta_0$ called distinguished vertices. A vertex $i \in \Delta_G$ belongs to $\Theta_0$ if the associated simple root is orthogonal to a maximal torus. Note that following our convention, projective homogeneous varieties of type $\Theta$ have a rational point if and only if $\Theta_0\subset \Theta$ (see \cite{tits} for the complete list of the Tits indices over fields).

Note that if $G_{an}$ is the anisotropic kernel of $G$, its Dynkin diagram $\Delta_{G_{an}}$ is obtained from $\Delta_G$ by removing the vertices which are not distinguished. The Tits index of $G$ is denoted by $\mathrm{Tits}(G)$ or $(\Delta_G,\Theta_0)$ depending on the context, and is represented by the Dynkin diagram of $G$, on which the vertices not contained in $\Theta_0$ are circled.

\begin{example}
    Let $n$ be an integer and $q$ be a non-degenerate $2n$-dimensional quadratic form over $F$, with trivial discriminant. Assuming $q$ is of Witt index $k$, denote by $q_{an}$ its anisotropic part and write $q=k\mathbb{H}~\bot~ q_{an}$. The orthogonal group $SO(q)$ is of type $D_n$ and its anisotropic kernel is isomorphic to the special orthogonal group of $q_{an}$. The distinguished vertices of $\Delta_{SO(q)}$ are $k+1,k+2,...,n$ so that $\mathrm{Tits}(SO(q))$ is represented as follows.  \\

{\centering\noindent\makebox[180pt]{

{

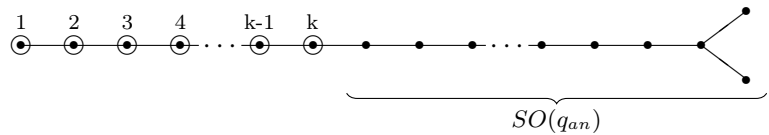
\begin{tabular}{c}

\begin{picture}(7,2)(0,0)

%ligne du dessus

\put(0,0){\circle*{3}}
\put(0,0){\line(1,0){20}}
\put(20,0){\circle*{3}}
 \put(20,0){\line(1,0){20}}
% \put(25,-2.5){\mbox{$\cdots$}}

\put(40,0){\circle*{3}}
\put(40,0){\line(1,0){20}}
\put(60,0){\circle*{3}}
\put(60,0){\line(1,0){7}}
\put(68.5,-2.7){\mbox{$\cdots$}}
\put(90,0){\circle*{3}}
\put(83,0){\line(1,0){92}}
\put(110,0){\circle*{3}}
\put(130,0){\circle*{3}}
\put(150,0){\circle*{3}}
\put(170,0){\circle*{3}}
\put(176.5,-2.7){\mbox{$\cdots$}}
\put(196,0){\circle*{3}}
\put(216,0){\circle*{3}}

\put(191,0){\line(1,0){65}}
\put(236,0){\circle*{3}}
\put(256,0){\circle*{3}}
\put(273,12.9){\circle*{3}}
\put(273,-12.9){\circle*{3}}

\put(256,0){\line(4,3){17}}
\put(256,0){\line(4,-3){17}}

%$p$-indices
\put(0,0){\circle{7}}
\put(20,0){\circle{7}}
\put(40,0){\circle{7}}
\put(60,0){\circle{7}}
\put(90,0){\circle{7}}
\put(110,0){\circle{7}}

%texte
\put(-2,6){{\scriptsize 1}}
\put(18,6){{\scriptsize 2}}
\put(38,6){{\scriptsize 3}}
\put(58,6){{\scriptsize 4}}
\put(84,6){{\scriptsize k-1}}
\put(107.5,6){{\scriptsize k}}

\end{picture}

 \begin{tikzpicture}[remember picture, overlay]

\draw [decorate,decoration={brace,amplitude=4pt,raise=1pt,mirror},xshift=162pt,yshift=20pt](-50pt,-37pt) -- (108pt,-37pt) node[black,midway,yshift=-0.4cm] {\footnotesize $SO(q_{an})$};

\end{tikzpicture}
\end{tabular}}}
}

\vspace{1.3cm}
\end{example}

\subsection{Weyl groups, double cosets and minimal coset representatives}

Denote by $W$ the Weyl group corresponding to a maximal torus $T$ in $G$. It is generated by the reflections $s_{\alpha}$ associated with the simple roots $\alpha\in\Sigma$. Given a subset $\Theta$ of $\Sigma$, we denote by $R_{\Theta}$ the set of roots generated by $\Theta$ and by $W_{\Theta}$ the subgroup generated by the reflections $s_{\alpha}$ with $\alpha\in \Theta$. It is the Weyl group of the root system generated by $\Theta$.

\begin{remark}
   Given a semi-simple group $G$, the simple roots attached to the subset $\Theta_0$ of distinguished vertices in $\Delta_G$ is a set of simple roots for the anisotropic kernel of $G$. 
\end{remark}
The length $\ell(w)$ of an element $w\in W$ is the minimal number of elements in an expression $w=s_{\alpha_{1}}...s_{\alpha_{k}}$ of $w$, as a product of reflections with respect to the simple roots.

\begin{prop}[{\cite[Ex. 7.1]{bros},\cite{humph}}]\label{test}
    Let $G$ be a semi-simple group. Denoting $W$ the Weyl group of $G$ and $\Sigma$ a system of simple roots for $G$, the following holds :
    \begin{enumerate}
        \item Any double coset in $W_{I}\textbackslash W/W_{J}$ given by two subsets $I,J\subset \Sigma$ contains a unique element of minimal length (with respect to $\Sigma$).
        \item A $w$ in such double coset is minimal if it satisfies one of the following equivalent conditions: 
        \begin{enumerate}
            \item $\ell(ws_{\alpha})=\ell(w)+1$ for $\alpha\in J$ and $\ell(s_{\alpha}w)=\ell(w)+1$ for $\alpha\in I$.

            \item $w\alpha>0$ for $\alpha\in J$ and $w^{-1}\alpha>0$ for $\alpha \in I$.
        \end{enumerate}
    \end{enumerate}

    Picking minimal length elements for each double coset yields the canonical set of minimal length coset representatives for $W_{I}\textbackslash W/W_{J}$.
\end{prop}

\subsection{Motives and Bialynicki-Birula-Brosnan decompositions}

\subsubsection{Statement of Brosnan Theorem}

Our main reference with regards to construction and basic properties of Chow groups and motives is \cite{ekm}. In this note, we work over the category $\mathrm{CM}(F,\mathbb{Z})$ of Chow motives with integral coefficients, constructed from the category of correspondences by additive and pseudo-abelian completion. Following standard notation, we denote by $\mathbb{Z}\{i\}$ with $i\in \mathbb{Z}$ the Tate motives, that is, the Tate twists of the motive of the base field. Unfolding Brosnan's result for groups of inner type reads as follows.

\begin{theorem}[\!\!{\cite[Theorem 7.4]{bros}}]\label{main}
  Let $G$ be a semi-simple algebraic group of inner type and $\Theta\subset \Delta_G$ be a subset of the vertices of its Dynkin diagram. Denote by $(\Delta_G,\Theta_0)$ its Tits index and by $\mathcal{W}$ the set of minimal length coset representatives for $W_{\Theta_0}\textbackslash W/W_{\Theta}$. Then we have
  $$M(X_{\Theta,G})\simeq \bigoplus_{w\in \mathcal{W}}M(X_{\Theta_w,G_{an}})\{\ell(w)\}$$
  in $\mathrm{CM}(F,\mathbb{Z})$, where for $w\in \mathcal{W}$, $\Theta_w=\{\alpha\in \Theta_0, w^{-1}\alpha\in R_{\Theta}\}$.
\end{theorem}

\begin{remark}
    In the full version of his result, Brosnan allows to replace $\Theta_0$ by an arbitrary subset $I\subset \Delta_G$ containing $\Theta_0$. We stick in this note to the case where $I=\Theta_0$ for input simplicity and since it's enough to treat the general case (it leads to the finest motivic decompositions).
\end{remark}

\subsubsection{The case of split $G$}\label{split}

Even though Bialynicki-Birula-Brosnan motivic decompositions handle split semi-simple algebraic groups, complete motivic decompositions were already known in this case, by \cite{kock}. Let $G$ be a split semi-simple algebraic group and $X_{\Theta,G}$ be a projective $G$-homogeneous variety. As the anisotropic kernel of $G$ is here trivial, $W_{\Theta_0}$ is trivial as well so that Theorem \ref{main} gives 
$$M(X_{\Theta,G})\simeq \bigoplus_{w\in W^{\Theta}}\mathbb{Z}\{\ell(w)\},$$where $W^{\Theta}$ denotes the set of minimal representatives for the right cosets $W/W_{\Theta}$. Consider now the three polynomials $$P_{W}(t)=\sum_{w\in W}t^{\ell(w)},~P_{W^{\Theta}}(t)=\sum_{w\in W^{\Theta}}t^{\ell(w)}~\mbox{ and }~P_{W_{\Theta}}(t)=\sum_{w\in W_{\Theta}}t^{\ell(w)}.$$

The coefficients of $P_{W^{\Theta}}$ determine the above motivic decomposition of $X_{\Theta,G}$. Since an element $w\in W$ uniquely decomposes as a product $w=u\cdot v$, with $u\in W^{\Theta}$ and $v\in W_{\Theta}$, we also get $$P_{W^{\Theta}}(t)=\frac{P_{W}(t)}{P_{W_{\Theta}}(t)}.$$
Now by the Chevalley-Solomon formula, $P_{W}$ and $P_{W_{\Theta}}$ are completely determined by the underlying combinatorial data (see \cite{carter}) and are available in SageMath via  \texttt{W.degrees()}. Our algorithm thus forks and bypasses Weyl group enumeration in this situation for a much faster division of polynomials with integral coefficients. 

\section{The Algorithm}
\label{sec:algorithm}

\subsection{Overview}

The algorithm takes as input:
\begin{itemize}
  \item a Dynkin type (e.g.\ $E_6$, $B_4$),
  \item a subset $\Theta_0 \subseteq \{1,\dots,r\}$ (Tits index),
  \item a subset $\Theta \subseteq \{1,\dots,r\}$ (projective homogeneous variety type),
\end{itemize}
and produces as output the full list of terms in the decomposition of Theorem~\ref{main},
aggregated by motivic isomorphism class.

The algorithm proceeds in three main stages:

\begin{enumerate}
  \item \textbf{Enumeration of $W^{\Theta}$}: compute the set of right-minimal representatives.
  \item \textbf{Filtering for $W_{\Theta_0}\textbackslash W/W_{\Theta}$}: retain those that are also left-minimal for $W_{\Theta_0}$.
  \item \textbf{Aggregation}: for each $w \in W_{\Theta_0}\textbackslash W/W_{\Theta}$, compute $\Theta_w$, $\ell(w)$, and aggregate by motivic contribution.
\end{enumerate}

\subsection{Enumeration of right-minimal representatives}

The set $W^{\Theta}$ of right-minimal representatives is computed by a depth-first enumeration (DFS) of the Cayley graph of $W$, starting from $\mathrm{id}$. Recall that the Cayley graph has the elements of $W$ as vertices, with an edge between $w$ and $s_i \cdot w$ for each simple generator $s_i$. Since $\ell(w s_j) = \ell(w) \pm 1$ for any simple reflection $s_j$ (length changes by $\pm 1$ at each step), the Cayley graph is graded by length, and edges can only go up or down by one level.

The enumeration maintains a stack \texttt{q} of elements to process and a set \texttt{seen} of already-discovered elements, both initialized to $\{\mathrm{id}\}$. At each step, an element $w$ is popped from the stack, and for each simple generator $s_i$, the product $s_i \cdot w$ is immediately reduced to its right-minimal representative via \textsc{ReduceRight}. If this representative has not been seen before, it is added to both \texttt{seen} and the stack.

\begin{algorithm_env}[Enumeration of $W^{\Theta}$]\quad
\label{alg:enum}
\begin{lstlisting}[style=pseudocode, mathescape=true]
Input:  Weyl group W, generators s_1,...,s_r, parabolic subset $\Theta$
Output: list of all right-minimal representatives $W^{\Theta}$

function ReduceRight(w, $\Theta$):
    repeat
        for each j in $\Theta$:
            if length(w * s_j) < length(w):
                w := w * s_j
    until no change
    return w

queue := [identity]
seen  := {identity}
while queue is not empty:
    w := pop(queue)
    for each i in {1,...,r}:
        w' := ReduceRight(s_i * w, $\Theta$)
        if w' not in seen:
            add w' to seen and to queue
return seen
\end{lstlisting}
\end{algorithm_env}

The termination of \textsc{ReduceRight} is immediate: as soon as no generator $s_j$ with $j \in \Theta$ decreases the length, the element is guaranteed to be the unique minimum of its coset $w W_{\Theta}$. The outer DFS terminates when the queue empties, which happens after at most $|W^{\Theta}|$ iterations. The key point is that as soon as the queue is empty, \texttt{seen} contains all elements of $W^{\Theta}$. This follows from the uniqueness of the minimal representative in each right coset: if \textsc{ReduceRight}$(s_i \cdot w)$ is already in \texttt{seen} for all $i$ and all $w \in$ \texttt{seen}, then by induction on word length, every element of $W^{\Theta}$ is already in \texttt{seen}.

\subsection{The bi-minimality filter}

The key algorithmic improvement over a naive approach is the following observation.
Since every minimal element $w$ in a double coset $W_{\Theta_0} \backslash W / W_{\Theta}$ 
can be written as $w = a \cdot b \cdot c$ with $a \in W_{\Theta_0}$, $c \in W_{\Theta}$, 
and $\ell(w) = \ell(a) + \ell(b) + \ell(c)$, $b \cdot c$ is 
right-minimal in $b W_{\Theta}$, and $a \cdot b$ is left-minimal in $W_{\Theta_0} b$. It therefore suffices to filter $W^{\Theta}$ for the additional left-minimality condition, rather than performing a full iterative bi-minimal reduction on arbitrary elements of $W$.

The algorithm offers two different implementations of this test, relying on the equivalent conditions of Proposition \ref{test}, selectable at runtime. Neither implementation dominates the other in general: either can give better performance depending on the group type and $\Theta_0$.

\paragraph{Length-based test (a).}
For each $i \in \Theta_0$, compute $\ell(s_i w)$ and compare with $\ell(w)$. This is
straightforward and relies only on the length function of $W$.

\paragraph{Root-based test (b).}
The equivalent condition is: $w^{-1}(\alpha_i) > 0$ for all $i \in \Theta_0$. This avoids computing the length of the $s_i w$.

\subsection{Computation of $\Theta_w$ and aggregation}

For each minimal representative $w$ of a double coset in $W_{\Theta_0} \backslash W / W_{\Theta}$, the subset $\Theta_w \subseteq \Theta_0$ is determined by the condition $w^{-1}(\alpha_i) \in R_{\Theta}$. The reasoning for the split case (Section \ref{split}) implies that the rank of $w$, defined as
\[
  \mathrm{abs\_rank}(w) = \frac{|W_{\Theta_0}|}{|W_{\Theta_w}|},
\]
determines the absolute rank of the motive of $X_{\Theta_w,G}$, that is, the number of Tate motives in the complete motivic decomposition of $M(X_{\Theta_w,G})$ over a separable closure. It thus measures the size of the motivic summand contributed by $w$, and serves as a final consistency check.

Contributions are aggregated by triples $(\Theta_w, \ell(w), \mathrm{rank}(w))$: all minimal representatives sharing the same triple are grouped, and their count gives the multiplicity.

\subsection{Consistency checks}
Two consistency checks are built into the algorithm:
\begin{enumerate}
    \item The two implementations of the left-minimality test (length-based 
    and root-based) are mathematically equivalent and can be run in parallel 
    to cross-validate the output.
    \item The sum of the abstract ranks of all motivic summands in the output 
    decomposition, denoted \texttt{covered\_tate}, must equal \texttt{target\_tate} 
    $= |W|/|W_{\Theta}| = P_W(1)/P_{W_{\Theta}}(1)$, which is the absolute rank of the motive 
    $M(X_{\Theta, G})$. This quantity is computed directly as a ratio of two integers, without requiring the full polynomial quotient of \S\ref{split}.
\end{enumerate}

\subsection{Tate trace mode}

Let $G$ be a semi-simple algebraic group over $F$ and $X$ be a projective homogeneous $G$-variety. Following \cite{decqm}, the integral Tate trace $\mathrm{Tr}(X)$ of $X$ is defined as a maximal pure Tate summand of $M(X)\in \mathrm{CM}(F,\mathbb{Z})$. The $p$-local versions of these Tate traces allow to completely classify motives arising from projective homogeneous varieties with finite coefficients \cite[Theorem 4.3]{decqm}. Under some assumptions we now explain, the Bialynicki-Birula-Brosnan motivic decompositions (Theorem \ref{main}) determine the integral Tate traces.

In this subsection we assume $G$ satisfies the following property: any projective $G$-homogeneous variety has a rational point if and only if it has a $0$-cycle of degree $1$. This condition is known to hold in many cases, including spinor and orthogonal groups of quadratic forms, special and projective linear groups of central simple algebras. It holds in general over number fields, while the question for arbitrary fields remains open (we refer the reader to \cite{pari} for more details).

Consider now a motivic decomposition arising from Theorem \ref{main} :
$$M(X_{\Theta,G})\simeq \bigoplus_{w\in \mathcal{W}}M(X_{\Theta_w,G_{an}})\{\ell(w)\}.$$
Our choice for this statement of fixing $I=\Sigma_0$ in \cite[Theorem 7.4]{bros} implies that this decomposition involves projective homogeneous varieties for the anisotropic kernel of $G$. In particular, the integral Tate trace of $X$ can be recovered directly from this decomposition, thanks to the following well-known integral version of \cite[Lemma 2.2]{decqm}.

\begin{lemma}
Let $X$ be a projective homogeneous variety. The following are equivalent :
\begin{enumerate}
    \item[(i)] $X$ has a zero-cycle of degree $1$;
    \item[(ii)] $M(X)\in \mathrm{CM}(F,\mathbb{Z})$ contains a summand isomorphic to $\mathbb{Z}\{0\}$;
    \item[(iii)] $M(X)\in \mathrm{CM}(F,\mathbb{Z})$ contains a summand isomorphic to $\mathbb{Z}\{k\}$, for some $k\in \mathbb{Z}$.
\end{enumerate}
\end{lemma}

\begin{proof}
$(i)\Rightarrow (ii)$ is well-known and $(ii)\Rightarrow (iii)$ is obvious. As for $(iii)\Rightarrow (i)$, a direct summand $\mathbb{Z}\{k\}$ of $M(X)$ yields two correspondences $\alpha:\mathbb{Z}\{k\}\rightsquigarrow X$ and $\beta:X \rightsquigarrow \mathbb{Z}\{k\}$ defining the summand. From the definition of the composition of correspondences, the intersection of their pull-backs to $\mathrm{CH}(\mathrm{Spec}(F)\times X \times \mathrm{Spec}(F))$ is a zero-cycle of degree $1$.
\end{proof}
 
The previous lemma asserts that under our assumptions, a direct summand $M(X_{\Theta_w,G_{an}})\{\ell(w)\}$ involved in the above motivic decomposition has non-trivial Tate trace if and only if the variety $X_{\Theta_w,G_{an}}$ has a rational point. It follows that it can only happen here if $\Theta_w=\Theta_0$ and then,  $M(X_{\Theta_w,G_{an}})\{\ell(w)\}$ is simply the Tate motive $\mathbb{Z}\{\ell(w)\}$.

The Tate trace mode hence performs the main algorithm for Bialynicki-Birula-Brosnan motivic decompositions, retaining only summands associated with minimal representatives $w$ for double cosets of $W_{\Theta_0} \backslash W / W_{\Theta}$ such that $\Theta_w = \Theta_0$. Rather than computing $\Theta_w$ explicitly for each $w$ and then checking 
whether $\Theta_w = \Theta_0$, the algorithm tests this condition directly 
by checking $w^{-1}(\alpha_i) \in R_{\Theta}$ for each $i \in \Theta_0$ 
in turn, stopping at the first failure. Elements passing this test are then 
aggregated by Tate twist $\ell(w)$, yielding the multiplicity of the pure Tate summand $\mathbb{Z}\{\ell(w)\}$ in the Tate trace of $M(X_{\Theta,G})$.

\section{Complexity Analysis}

\subsection{Enumeration of $W^{\Theta}$}
The DFS enumeration of $W^{\Theta}$ via the Cayley graph of $W$ finds each element of 
$W^{\Theta}$ exactly once, via the test \texttt{w' not in seen}. The cost per 
element is dominated by the call to \textsc{ReduceRight}, which performs at most 
$|\Theta| \cdot h$ length comparisons, where $h$ is the Coxeter number of $W$. 
Since $|W^{\Theta}| = |W|/|W_{\Theta}|$, the total time cost of the enumeration is 
$O(|W| \cdot r \cdot |\Theta| \cdot h / |W_{\Theta}|)$, where $r$ is the rank of $W$. 
The set \texttt{seen} must store all discovered elements simultaneously, which 
can become significant for large groups. The optional \texttt{--words} flag 
controls whether sample reduced words are retained in the output; disabling 
it (default run) reduces the memory footprint.

\subsection{Filtering for bi-minimality}
The left-minimality filter applies one of the two equivalent tests of 
Proposition~\ref{test} to each element of $W^{\Theta}$, performing at most 
$|\Theta_0|$ operations per element, with a short-circuit at the first failure. 
The total filtering cost is dominated by the enumeration cost of the previous step.

\subsection{Comparison with naive bi-minimal reduction}
The naive approach of performing bi-minimal reduction on all elements of $W$ 
costs $O(|W| \cdot r \cdot h)$ up to factors depending on $|\Theta|$ and 
$|\Theta_0|$, which is a factor of $|W_{\Theta}|$ worse than the present 
algorithm. The improvement is significant: for example, if $G$ has type $E_8$ 
and $\Theta = \{1,2,3,4,5,6,7\}$, then $|W_{\Theta}| = |W(E_7)| = 2\,903\,040$, 
reducing the enumeration cost by a factor of nearly three million. More generally, the complexity of the present algorithm is naturally 
calibrated to the size of the motive investigated itself, as
$|W^{\Theta}| = |W|/|W_{\Theta}|$ is precisely the absolute rank of 
$M(X_{\Theta,G})$.

\section{Implementation in SageMath}

\subsection{Software environment}

The algorithm is implemented as a SageMath version 9.5 script (\texttt{BBBdec.py}). SageMath provides the required algebraic primitives: Cartan types (\texttt{CartanType}), root systems and root lattices (\texttt{RootSystem}, \texttt{root\_lattice()}), Weyl groups and their elements
(\texttt{WeylGroup}, \texttt{simple\_reflection}, \texttt{length}, \texttt{reduced\_word}),
and polynomial arithmetic (\texttt{PolynomialRing}, \texttt{quo\_rem}). Weyl group elements are represented by their reduced words. SageMath's built-in \texttt{length()} and \texttt{w.action(alpha)} methods compute the length of a Weyl group element and their actions on a root $\alpha$.

\subsection{Availability}
The SageMath script \texttt{BBBdec.py} is available on the following webpage, displaying several examples: \url{https://www.math.univ-paris13.fr/~declercq/BBBdec.html}

\subsection{Command-line interface}

The script exposes a command-line interface with the following main parameters:

\begin{center}
\begin{tabular}{>{\ttfamily}ll}
\toprule
\normalfont Flag & Meaning \\
\midrule
--dynkin TYPE    & Dynkin type of $G$ (e.g.\ \texttt{A4}, \texttt{D5}, \texttt{E6}) \\
--Sigma0 LIST    & Comma-separated vertices of $\Theta_0$ (anisotropic kernel)\\
--J LIST         & Comma-separated vertices of $\Theta$ (type of the variety)\\
--Sigma0comp LIST & Complement of \texttt{Sigma0} (alternative input) \\
--Jcomp LIST     & Complement of \texttt{J} (alternative input) \\
--tatetrace      & Compute Tate trace only\\
--use-root-test  & Use root-based left-minimality test \\
--progress       & Print progress during enumeration \\
\bottomrule
\end{tabular}
\end{center}

\section{Explicit examples}

From now on to simplify notation, we switch to the opposite of Tits convention for the types of projective homogeneous varieties  and $\Theta_0$ (see Remark \ref{convent}). With this new convention, for instance, the Borel variety of a semi-simple group $G$ is of type $\Delta_G$, or the projective quadric associated with a quadratic form is of type $1$ (and not of type $\Delta_{SO(q)}\setminus \{1\}$). Note that the output of \texttt{BBBdec.py} uses this convention, and we input with this convention as well using \texttt{Jcomp}.\\

\begin{example}[{\cite[Example 7.6]{bros}}]Let $q$ be a 12-dimensional non-degenerate quadratic form over $F$ of Witt index $1$ with anisotropic part $q_{an}$ (so that $\Delta_{SO(q_{an})}=\Theta_0=\Delta_{SO(q)}\setminus \{1\}$). The orthogonal Grassmannian variety $X(2;q)$ of isotropic planes with respect to $q$ is projective $SO(q)$-homogeneous of type $\{2\}$. Running \texttt{BBBdec.py} with \texttt{--words} displays the set of minimal representatives for double cosets. To avoid enumerating the whole $\Theta_0$ we input its complement \texttt{Sigma0comp} and we input the subset $\{2\}$ with \texttt{Jcomp}.\\

\begin{lstlisting}[style=pseudocode,numbers=none,basicstyle=\ttfamily\footnotesize]
sage BBBdec.py \
--dynkin=D6 --Sigma0comp=1 --Jcomp=2 --words

=== Bialynicki-Birula-Brosnan decomposition (Brosnan Thm. 7.4) ===
dynkin=['D', 6], Sigma0=[2, 3, 4, 5, 6], J=[1, 3, 4, 5, 6]

=== AGGREGATED TERMS ===
mult=1| Levi_type=D5 | Jwcomp=[2] | twist=0 |sample_w=1
mult=1| Levi_type=D5 | Jwcomp=[3] | twist=2 |sample_w=s1*s2
mult=1| Levi_type=D5 | Jwcomp=[2] | twist=9 |sample_w=s1*s2*s3*s4*s6*s5*s4*s3*s2

=== STATS ===
covered_tate: 60
target_tate: 60
\end{lstlisting}

We thus get the motivic decomposition
$$M(X(2;q)) \simeq M(X)\{0\} \oplus M(Y)\{2\} \oplus M(X)\{9\}$$
where $X$ and $Y$ are as follows : as the vertex labelled $[2]$ in $\Delta_{SO(q)}$ corresponds to vertex $1$ of $\Delta_{SO(q_{an})}$, $X=X_{1,SO(q_{an})}$ is the projective quadric associated with the anisotropic part of $q$. In the same way, the variety $Y$ corresponds to the vertex labelled $[3]$ in $\Delta_{SO(q)}$, hence $Y=X_{2,SO(q_{an})}$ is the variety of isotropic planes $X(2;q_{an})$ with respect to $q_{an}$. These motivic summands are respectively associated with the Weyl group elements $\mathrm{id}$, $s_1s_2$ and $s_1s_2s_3s_4s_6s_5s_4s_3s_2$.\\
\end{example}

\begin{example}[Tate trace mode]Let $G$ be a semi-simple group of type $E_6$, with $\Theta_0=\Delta_G\setminus \{1,6\}$.\vspace{0.1cm}

\noindent
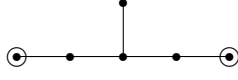
\begin{minipage}{.3\textwidth}%
\hspace{1cm}
\scalebox{1}{
\begin{picture}(7,2)(0,0)

%ligne du dessus

\put(0,0){\circle*{3}}
\put(0,0){\line(1,0){20}}
\put(20,0){\circle*{3}}
 \put(20,0){\line(1,0){20}}
% \put(25,-2.5){\mbox{$\cdots$}}

\put(40,0){\circle*{3}}
\put(40,0){\line(1,0){20}}
\put(60,0){\circle*{3}}
\put(60,0){\line(1,0){20}}
\put(80,0){\circle*{3}}
\put(40,0){\line(0,1){20}}
\put(40,20){\circle*{3}}

\put(0,0){\circle{7}}
\put(80,0){\circle{7}}
\end{picture}
}
\end{minipage}%
\hfill
\begin{minipage}{.65\textwidth}%
The anisotropic kernel $G_{an}$ of $G$ is of type $D_4$, and the variety $X_{\{1,6\},G}$ associated with the subset $\Theta=\{1,6\}$ of $\Delta_G$ has a non-trivial Tate trace,  since $\Theta$ is circled in the Tits index of $G$.
\vspace{0.3cm}
\end{minipage}%
\begin{lstlisting}[style=pseudocode,numbers=none,basicstyle=\ttfamily\footnotesize]
sage BBBdec.py \
--dynkin=E6 --Sigma0comp=1,6 --Jcomp=1,6 --tatetrace

=== TATE TRACE ===
  mult=1 | twist=0  | contrib=1
  mult=2 | twist=8  | contrib=2
  mult=2 | twist=16 | contrib=2
  mult=1 | twist=24 | contrib=1
\end{lstlisting}
Hence $\mathrm{Tr}(X_{\{1,6\},G})= \mathbb{Z}\{0\}\oplus\mathbb{Z}\{8\}^{\oplus 2}\oplus\mathbb{Z}\{16\}^{\oplus 2}\oplus\mathbb{Z}\{24\}$\\
\end{example}

\begin{example}[Type $E_8$]Let $G$ be a semi-simple group of type $E_8$, with $\Theta_0=\Delta_G\setminus \{1\}$. \vspace{0.1cm}

\noindent
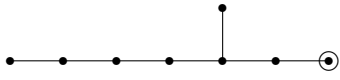
\begin{minipage}{.3\textwidth}%
\hspace{0.3cm}
\scalebox{1}{
\begin{picture}(7,2)(0,0)

%ligne du dessus

\put(0,0){\circle*{3}}
\put(0,0){\line(1,0){20}}
\put(20,0){\circle*{3}}
 \put(20,0){\line(1,0){20}}
% \put(25,-2.5){\mbox{$\cdots$}}

\put(40,0){\circle*{3}}
\put(40,0){\line(1,0){20}}
\put(60,0){\circle*{3}}
\put(60,0){\line(1,0){20}}
\put(80,0){\circle*{3}}
\put(80,0){\line(1,0){20}}
\put(100,0){\circle*{3}}
\put(100,0){\line(1,0){20}}
\put(120,0){\circle*{3}}
\put(80,0){\line(0,1){20}}
 \put(80,20){\circle*{3}}

\put(120,0){\circle{7}}
\end{picture}
}
\end{minipage}%
\hfill
\begin{minipage}{.65\textwidth}%
The anisotropic kernel $G_{an}$ of $G$ is of type $D_7$. Consider the variety $X_{8,G}$ associated with the 8th vertex of $\Delta_G$. Note that as $|W(G)|=696729600$, a brute-force bi-minimal reduction to compute the decomposition of $M(X_{8,G})$ is not reasonable.
\vspace{0.3cm}
\end{minipage}%
\begin{lstlisting}[style=pseudocode,numbers=none,basicstyle=\ttfamily\footnotesize]
sage BBBdec.py \
--dynkin=E8 --Sigma0comp=1 --Jcomp=8 --use-root-test

=== Bialynicki-Birula-Brosnan decomposition (Brosnan Thm. 7.4) ===
dynkin=['E', 8], Sigma0=[2, 3, 4, 5, 6, 7, 8], J=[1, 2, 3, 4, 5, 6, 7]

=== AGGREGATED TERMS ===
  mult=1| Levi_type=D7 | Jwcomp=[8] | twist=0  | abs_rank=14 | contrib=14
  mult=1| Levi_type=D7 | Jwcomp=[2] | twist=7  | abs_rank=64 | contrib=64
  mult=1| Levi_type=D7 | Jwcomp=[7] | twist=18 | abs_rank=84 | contrib=84
  mult=1| Levi_type=D7 | Jwcomp=[3] | twist=29 | abs_rank=64 | contrib=64
  mult=1| Levi_type=D7 | Jwcomp=[8] | twist=45 | abs_rank=14 | contrib=14

=== STATS ===
total_W: 696729600
WI_order: 322560
WJ_order: 2903040
WJ_reps: 240
bimin_reps: 5
aggregated_classes: 5
covered_tate: 240
target_tate: 240
\end{lstlisting}
This computation takes approximately 250 seconds on a standard computer. 
Translating \texttt{Jwcomp} indices to vertices of $\Delta_{G_{an}}$, 
we get: $$M(X_{8,G})\simeq M(X_{1,G_{an}})\{0\}\oplus M(X_{6,G_{an}})\{7\} \oplus M(X_{2,G_{an}})\{18\} \oplus M(X_{7,G_{an}})\{29\}\oplus M(X_{1,G_{an}})\{45\}$$
\end{example}

\section{Application: motivic decompositions in characteristic $2$.}

In characteristic different from $2$, \cite[Corollary 15.14]{kar} describes motivic decompositions of projective homogeneous varieties for semi-simple algebraic groups of classical type, with respect to the underlying algebraic structures (namely, central-simple algebras with involution). Our algorithm can be used to help showing that Bialynicki-Birula-Brosnan decompositions extend these decompositions to arbitrary characteristic (replacing algebras with orthogonal involutions by quadratic pairs for type $D_n$).\\

\noindent \textbf{Notation (see \cite{kmrt}).} Let $F$ be a field and $(A,\sigma)$ be a central-simple $F$-algebra with symplectic involution. Let $D$ a division algebra Brauer-equivalent to $A$, $V$ a right $D$-module with $h:V\times V\longrightarrow D$ the hermitian form adjoint to $(A,\sigma)$. Denote by $X(j;(A,\sigma))$ the variety of $\sigma$-isotropic ideals in $A$ of reduced dimension $j$. We say that $(A,\sigma)$ is isotropic if it contains a non-trivial isotropic ideal. In this case, there is an orthogonal decomposition $(V,h)=\mathbb{H}(D)~\bot ~(W,h')$ and we denote by $(B,\tau)$ an algebra with symplectic involution adjoint to $(W,h')$. Following the usual notations, we denote by $e_1,..,e_n$ the canonical basis of the underlying vector space of our root system of type $C_n$, and set $\alpha_i=e_i-e_{i+1}$ ($1\leq i\leq n-1$), $\alpha_n=2e_n$ for the simple roots.\\

The next result for semi-simple groups of type $C_n$ in characteristic $2$ is a key tool of \cite{decqmdb}.

\begin{theorem}
    Let $F$ be a field of characteristic $2$ and $(A,\sigma)$ be a central-simple $F$-algebra with symplectic involution of degree $2n$, $n\geq 4$. Assume that $(A,\sigma)$ is isotropic and $A$ is Brauer-equivalent to a quaternion algebra $D$. The following decomposition holds :
    \begin{align*}M(X(2;(A,\sigma)))=\mathbb{Z}\{0\}\oplus M(\mathrm{SB}(D)\!\times \! X(1;(B,\tau))\{1\}\oplus M(\mathrm{SB}(D))\{{ 2n-3 }\}\\\oplus M(X(2;(B,\tau)))\{4\}
\oplus M(\mathrm{SB}(D)\!\times \!X(1;(B,\tau)))\{2n-2\}\oplus\mathbb{Z}\{4n-5\}.
\end{align*}
where $\mathrm{SB}(D)$ denotes the Severi-Brauer variety of $D$.
\end{theorem}

\begin{proof}We are looking for elements in the set of minimal length representatives for the double-cosets $W_{\{1,3,...n\}}\! \setminus W / W_{\{1,3,...,n\}}$, where $W$ is the Weyl group of the automorphism group of $(A,\sigma)$ (of type $C_n$, $n\geq 4$). We show that elements $e$, $s_2$, $s_2s_3s_1s_2$ and $s_2s_3 ...s_{n-1}s_ns_{n-1}...s_3s_2$ (extrapolated from the algorithm) provide enough summands to get the complete motivic decomposition.

\begin{enumerate}
    \item $e$ is obviously of minimal length in its double coset and gives rise to a Tate summand $\mathbb{Z}\{0\}$ since we have $\Theta_{e}=\{\alpha_1,\alpha_3,...,\alpha_n\}$.

    \item $s_2$ is of minimal length, otherwise we would have $e\in W_{\{1,3,...n\}}\! \setminus s_2 / W_{\{1,3,...,n\}}$ and $s_2$ would belong to $W_{\{1,3,...n\}}$. A direct computation shows that $s_2(\alpha_1)=\alpha_1+\alpha_2$, $s_2(\alpha_3)=\alpha_2+\alpha_3$ and $s_2$ acts trivially on $\alpha_i$ if $i\geq 4$ so that $\Theta_{s_2}=\{\alpha_4,...,\alpha_n\}$. 

    \item We show that $w=s_2s_3s_1s_2$ is minimal with Proposition \ref{test}. First $w(\alpha_1)=w^{-1}(\alpha_1)=\alpha_3$, $w(\alpha_3)=w^{-1}(\alpha_3)=\alpha_1$ and $w(\alpha_4)=w^{-1}(\alpha_4)$ (which is $2e_2=2\alpha_2+2\alpha_3+\alpha_4$ if $n=4$ and $e_2-e_5=\alpha_2+\alpha_3+\alpha_4$ if $n> 4$) are indeed positive. Since for $i> 4$, $w$ acts trivially on $\alpha_i$, $s_2s_3s_1s_2$ is of minimal length in its double coset and does not belong to the previous ones, by unicity. These readily give $\Theta_w=\{\alpha_1,\alpha_3,\alpha_5,...,\alpha_n\}$

    \item We use Proposition \ref{test} again for $\tilde{w}=s_2s_3 ...s_{n-1}s_ns_{n-1}...s_3s_2$ (note that it is a palindrome). One easily see that $\tilde{w}(e_2)=-e_2$ and $\tilde{w}$ acts trivially on $e_i$, for $i\neq 2$. Thus we have $\tilde{w}(\alpha_1)=e_1+e_2=\alpha_1+2\sum_{2\leq k \leq n-1} \alpha_k + \alpha_n$ is positive, and for $i\geq 3$, $\tilde{w}$ acts trivially on $\alpha_i$. It follows that $\Theta_{\tilde{w}}=\{\alpha_3,\alpha_4,...,\alpha_n\}$
    \end{enumerate}

These elements respectively give rise to direct summands $\mathbb{Z}\{0\}$, $M(\mathrm{SB}(D)\!\times \! X(1;(B,\tau))\{1\}$, $M(X(2;(B,\tau)))\{4\}$ and $M(\mathrm{SB}(D))\{2n-3\}$ of the motive $M(X(2;(A,\sigma)))$. The push-forward of the exchange isomorphism on the Chow group $\mathrm{CH}_{\ast}(X(2;(A,\sigma))\times X(2;(A,\sigma)))$, applied to the projectors of the first two summands, yields by symmetry two new summands $\mathbb{Z}\{4n-5\}$ and $M(\mathrm{SB}(D)\!\times \! X(1;(B,\tau))\{2n-2\}$ (note that the Tate twists come from $\dim(X(2;(A,\sigma))=4n-5$).

The sum of the absolute ranks of these summands is equal to the rank of the whole motive $M(X(2;(A,\sigma)))$, so that this decomposition is complete (these computation is independent of the characteristic of $F$, and one may use \cite[Corollary 15.14]{kar} here as well). 

\begin{remark}
    Extrapolating from the low rank results given by the algorithm, one may as well show that the two summands $\mathbb{Z}\{4n-5\}$ and $M(\mathrm{SB}(D)\!\times \! X(1;(B,\tau))\{2n-2\}$ obtained in the above proof by symmetry are respectively given by the words $s_2s_3...s_ns_1s_2s_3...s_n(s_{n-2}s_{n-1})...(s_2s_3)(s_1s_2)$ and $s_2s_3...s_{n-1}s_ns_{n-1}...s_4s_3s_1s_2$.
\end{remark}

\end{proof}

\printbibliography

\end{document}